\documentclass{amsart}
\usepackage{macros}
\standardsettings
\colorcommentstrue

\draftfalse

\begin{document}
\title{Variations on Dirichlet's theorem}

\authorlior\authordavid


\maketitle

\let\temp\mathbb
\let\mathbb\amsbb

\begin{abstract}
We give a necessary and sufficient condition for the following property of an integer $d\in\N$ and a pair $(a,A)\in\mathbb R^2$: There exist $\kappa > 0$ and $Q_0\in\N$ such that for all $\mathbf x\in \mathbb R^d$ and $Q\geq Q_0$, there exists $\mathbf p/q\in\mathbb Q^d$ such that $1\leq q\leq Q$ and $\|\mathbf x - \mathbf p/q\| \leq \kappa q^{-a} Q^{-A}$. This generalizes Dirichlet's theorem, which states that this property holds (with $\kappa = Q_0 = 1$) when $a = 1$ and $A = 1/d$. We also analyze the set of exceptions in those cases where the statement does not hold, showing that they form a comeager set. This is also true if $\mathbb R^d$ is replaced by an appropriate ``Diophantine space'', such as a nonsingular rational quadratic hypersurface which contains rational points. Finally, in the case $d = 1$ we describe the set of exceptions in terms of classical Diophantine conditions.
\end{abstract}

\let\amsbb\mathbb
\let\mathbb\temp

\section{Introduction}

Fix $d\in\N$. Dirichlet's theorem states that for every $\xx\in\R^d$ and $Q \geq 1$ there exists $\pp/q\in\Q^d$ 
\comdavid{I deleted the footnote because the footnote mark was interfering with the inline math. There doesn't seem to be anywhere to place the footnote in the sentence which is more appropriate. But anyway, I don't think the footnote is necessary.} with $1\leq q\leq Q$ such that
\begin{equation}
\label{strong}
\left\|\xx - \frac{\pp}{q}\right\| < q^{-1}Q^{-1/d},
\end{equation}
where $\left\| \cdot \right\|$ denotes the max norm in $\R^d$. We recall that an immediate corollary to this theorem is that for every irrational $\xx\in\R^d\butnot\Q^d$, there exist infinitely many reduced rationals $\pp/q\in\Q^d$ satisfying
\[
\left\|\xx - \frac{\pp}{q}\right\| < q^{-(1+1/d)}.
\]
In this paper, we will refer to this statement as ``Dirichlet's corollary'', to distinguish it from ``Dirichlet's theorem'' which is the original statement \eqref{strong}.

Much of classical Diophantine approximation theory can be understood as an attempt to understand when and how Dirichlet's corollary can be improved. For example, we recall that if $\psi:\N\to (0,\infty)$, then a point $\xx\in\R^d$ is said to be \emph{$\psi$-approximable} if there exist infinitely many reduced rationals $\pp/q\in\Q^d$ satisfying
\[
\left\|\xx - \frac{\pp}{q}\right\| < \psi(q).
\]
For each $c\in\R$, we use the notation $\psi_c(q) := q^{-c}$. Then Dirichlet's corollary states that every irrational vector in $\R^d$ is $\psi_{1 + 1/d}$-approximable.

A natural analogue of the notion of $\psi$-approximability which is capable of expressing Dirichlet's theorem is as follows:
\begin{definition}
\label{definitionPsiapproximable}
Given $\Psi:\N\times\N\to (0,\infty)$ and $Q_0\in\N$, we will say that a point $\xx\in\R^d$ is \emph{$(\Psi,Q_0)$-approximable} if for all $Q\geq Q_0$, there exists $\pp/q\in\Q^d$ and $1\leq q\leq Q$ such that
\begin{equation}
\label{Psiapproximable}
\left\|\xx - \frac{\pp}{q}\right\| < \Psi(q,Q).
\end{equation}
If $\xx$ is $(\Psi,Q_0)$-approximable for some $Q_0\in\N$, then we say that $\xx$ is \emph{$\Psi$-approximable}. If every vector in a set $S$ is $(\Psi,Q_0)$-approximable with $Q_0$ uniform over $S$, then we say that $S$ is \emph{uniformly $\Psi$-approximable}; if $Q_0$ is not uniform, then we simply say that $S$ is \emph{$\Psi$-approximable}.

For each $a,A\in\R$, we use the notation $\Psi_{a,A}(q,Q) = q^{-a} Q^{-A}$. Then Dirichlet's theorem states that every vector in $\R^d$ is $(\Psi_{1,1/d},1)$-approximable.
\end{definition}

\begin{remark*}
One might guess that if $\Psi(q,Q) = \psi(q)$, then the set of $\Psi$-approximable points is the same as the set of $\psi$-approximable points. However, in this scenario, if $\psi(1) = 1$, then every point is considered to be $\Psi$-approximable (since $q = 1$ satisfies \eqref{Psiapproximable}), regardless of $\psi$-approximability. By making some changes to Definition \ref{definitionPsiapproximable}, one could rig it so that $\Psi$-approximability is in fact equivalent to $\psi$-approximability when $\Psi(q,Q) = \psi(q)$. However, there seems to be little point in doing so, and it would make the statement of our main theorem less elegant, so we use the definition as stated.
\end{remark*}

A special case of $\Psi$-approximability which has been considered in the literature is \emph{Dirichlet improvability} \cite{DavenportSchmidt1}: a point $\xx\in\R^d$ is said to be \emph{Dirichlet improvable} if it is $\alpha\Psi_{1,1/d}$-approximable for some $0 < \alpha < 1$. Moreover, $\xx$ is said to be \emph{singular} \cite{Baker1, Baker2} if this condition holds for every $0 < \alpha < 1$. However, until now the following natural question has not been answered: For which pairs $(a,A)\in\R^2$ is every point $\Psi_{a,A}$-approximable? Our main result is a complete answer to this question up to a multiplicative constant:

\begin{theorem}
\label{theorem1}
Fix $d\in\N$, and let $f_d:\R\to\R$ be given as follows:
\begin{equation}
\label{fddef}
f_d(a) = \begin{cases}
1 + |a| & a \leq 0\\
\left(\sum_{i = 0}^{d - 1} a^i\right)^{-1} & 0 \leq a \leq 1\\
1 + 1/d - a & 1 \leq a \leq 1 + 1/d\\
0 & a \geq 1 + 1/d
\end{cases}.
\end{equation}
Fix $(a,A)\in\R^2$.
\begin{itemize}
\item[(i)] If $A < f_d(a)$, then for every $\epsilon > 0$, $\R^d$ is uniformly $\epsilon\Psi_{a,A}$-approximable.
\item[(ii)] If $A > f_d(a)$, then $\R^d$ is not uniformly $\kappa\Psi_{a,A}$-approximable for any $\kappa > 0$.
\item[(iii)] If $A = f_d(a)$, then there exist $0 < \epsilon \leq \kappa < \infty$ such that $\R^d$ is uniformly $\kappa\Psi_{a,A}$-approximable but not uniformly $\epsilon\Psi_{a,A}$-approximable.
\end{itemize}
\end{theorem}
Note that the function $f_d$ given by \eqref{fddef} is continuous and nonincreasing. We now give the proof of Theorem \ref{theorem1} assuming that it has been verified for the case $a\in [0,1]$, $A = f_d(a)$:
\begin{proof}
First we observe that uniform $\Psi$-approximability has the following comparison properties:
\begin{itemize}
\item[1.] If $0 < \epsilon \leq \kappa < \infty$ and $A_1 < A_2$, then every uniformly $\kappa\Psi_{a,A_2}$-approximable set is uniformly $\epsilon\Psi_{a,A_1}$-approximable.
\item[2.] If $a_1 + A_1 = a_2 + A_2$ and $a_1 < a_2$, then every uniformly $\alpha\Psi_{a_1,A_1}$-approximable set is uniformly $\alpha\Psi_{a_2,A_2}$-approximable.
\end{itemize}
Property (1) allows us to reduce cases (i) and (ii) of Theorem \ref{theorem1} to case (iii), while property (2) allows us to omit the verification along the segment $a\in (1,1 + 1/d)$. Along the segments $a\in \OC{-\infty}0$ and $a\in\CO{1 + 1/d}\infty$, the uniform $\Psi_{a,A}$-approximability of $\R^d$ is verified by setting $q = Q$ and $q = 1$ in Definition \ref{definitionPsiapproximable}, respectively, and then choosing $\pp$ so as to minimize $\|\xx - \pp/q\|$. The proof that $\R^d$ is not uniformly $\epsilon\Psi_{a,A}$-approximable when $a\in (-\infty,0)$ can be omitted due to property (2). The existence of badly approximable points implies that $\R^d$ is not $\epsilon\Psi_{a,A}$-approximable when $a\in\CO{1 + 1/d}\infty$. So we are reduced to proving (iii) along the segment $a\in[0,1]$; for this, see Section \ref{sectiontheorem1}.
\end{proof}

Given Theorem \ref{theorem1}, a natural question is whether the negative results in cases (ii) and (iii) can be improved by replacing ``uniformly $\Psi$-approximable'' by just ``$\Psi$-approximable''. In case (iii), this appears to be a delicate issue, but in case (ii) we have the following answer, which also shows that the set of exceptions to $\Psi$-approximability is in a sense ``large'':

\begin{theorem}
\label{theorem2}
Fix $d\in\N$ and $(a,A)\in\R^2$ such that $A > f_d(a)$. Then the set
\[
\{\xx\in\R^d : \text{$\xx$ is not $\kappa\Psi_{a,A}$-approximable for any $\kappa > 0$}\}
\]
is comeager.
\end{theorem}

The deduction of Theorem \ref{theorem2} from Theorem \ref{theorem1}(ii) actually holds in a great degree of generality; see Section \ref{sectiongeneral} for details.

\begin{remark*}
Theorem \ref{theorem2} illustrates a difference between the theory of $\psi$-approximability and $\Psi$-approximability: while the set of $\psi$-approximable points is always comeager, the set of $\Psi$-approximable points can be meager if $\Psi$ decays sufficiently quickly.
\end{remark*}

In the case $d = 1$, we further investigate the set of exceptions to $\Psi$-approximability via a kind of ``inverse duality principle'' reminiscent of the equality between the set of Dirichlet improvable points and the set of badly approximable points \cite{DavenportSchmidt1}:

\begin{theorem}
\label{theorem3}
Fix $(a,A)\in\R^2$ such that $a < 1 < \min(A,A + a)$, and let
\begin{align*}
b &= \min(A,A + a) - 1,&
c &= (A - |a|)/b.
\end{align*}
Suppose that $c > 2$. Then there exists a constant $C > 0$ such that the following implications hold for $x\in\R$ and $\alpha > 0$:
\begin{itemize}
\item[(i)] If $x$ is $\alpha\Psi_{a,A}$-approximable, then $x$ is not $(C\alpha^b)^{-1}\psi_c$-approximable.
\item[(ii)] If $x$ is not $\alpha\Psi_{a,A}$-approximable, then $x$ is $C\alpha^{-b}\psi_c$-approximable.
\end{itemize}
When $c = 2$, the conclusion holds if $\alpha$ is sufficiently small.
\end{theorem}

We remark that the case $d = 1$, $A + a < 2$ of Theorem \ref{theorem2} follows from Theorem \ref{theorem3}, since the set of Liouville numbers is comeager. On the other hand, since the set of very well approximable numbers is a Lebesgue nullset, it follows that the set considered in Theorem \ref{theorem3} is a Lebesgue nullset whenever $d = 1$ and $A + a < 2$.

{\bf Acknowledgements.} The first-named author was supported in part by the Simons Foundation grant \#245708.

\section{Proof of Theorem \ref{theorem1}}
\label{sectiontheorem1}

\begin{convention*}
\label{conventionimplied}
The symbols $\lesssim_\times$, $\gtrsim_\times$, and $\asymp_\times$ will denote coarse multiplicative asymptotics. For example, $A\lesssim_\times B$ means that there exists a constant $C > 0$ (the \emph{implied constant}) such that $A\leq C B$. It is understood that the implied constant $C$ is only allowed to depend on certain ``universal'' parameters, to be understood from context.
\end{convention*}

In this section, we fix $d\in\N$ and $a\in[0,1]$, and we let
\[
A = f_d(a) = \frac{1}{\sum_{i = 0}^{d - 1} a^i}\cdot
\]
We prove the existence of $0 < \epsilon \leq \kappa < \infty$ such that (iii) of Theorem \ref{theorem1} holds, thus completing the proof of Theorem \ref{theorem1}. The proof of the existence of $\epsilon$ provides the clearest intuition for why the formula for $f_d(a)$ ($a\in[0,1]$) naturally appears in this context.

\begin{proof}[Existence of $\epsilon$]
Let
\begin{align*}
\alpha_j &= A \sum_{i = 0}^{j - 1} a^i \;\;\;\; (j = 0,\ldots,d),
\end{align*}
and note that
\[
\alpha_i = A + a\alpha_{i - 1}, \;\; \alpha_0 = 0, \;\; \alpha_d = 1.
\]
Now fix $Q\in\N$ large to be determined, and let
\begin{align*}
n_i &= \lceil Q^{\alpha_i - \alpha_{i - 1}}\rceil & (i = 1,\ldots,d)\\
Q_j &= 2\prod_{i = 1}^j n_j & (j = 0,\ldots,d)
\end{align*}
so that
\begin{align*}
Q_0\divides Q_1\divides \cdots \divides Q_d\\
Q_i \asymp_\times Q^{\alpha_i}.
\end{align*}
Now let
\[
\xx = \left(\frac1{Q_1},\ldots,\frac1{Q_d}\right) \in\R^d.
\]
Let $\pp/q\in\Q^d$ be a rational such that $1\leq q\leq Q$. Since $Q_d > 2Q \geq 2q \geq 2 = Q_0$, there exists $i = 1,\ldots,d$ such that $Q_{i - 1} \leq 2q < Q_i$. Then
\[
qx_i = \frac{q}{Q_i} < \frac12
\]
and thus since $p_i\in\Z$,
\[
\left\|\xx - \frac\pp q\right\| \geq \frac1q |qx_i - p_i| \geq \frac{1}{q} qx_i = \frac{1}{Q_i}\cdot
\]
So
\[
q^a Q^A \left\|\xx - \frac\pp q\right\| \geq (Q_{i - 1}/2)^a Q^A/Q_i \asymp_\times Q^{A + a\alpha_{i - 1} - \alpha_i} = 1,
\]
i.e. $\|\xx - \pp/q\| \geq \epsilon \Psi_{a,A}(q,Q)$, where $\epsilon$ is the reciprocal of the implied constant.
\end{proof}

\begin{proof}[Existence of $\kappa$]
Fix $\kappa > 0$ large to be determined, and fix $\xx\in\R^d$ and $Q\geq Q_0 := 1$. For each $q = 1,\ldots,Q$, let $\xx_q$ denote the element of $q\xx + \Z^d$ which minimizes $\|\xx_q\|$. By contradiction, suppose that
\begin{equation}
\label{contradictionhypothesis}
\|\xx_q\| \geq \kappa q^{1 - a} Q^{-A} \all q = 1,\ldots,Q.
\end{equation}
\begin{claim}
\label{claimminkowski}
There exists a sequence $(q_i)_0^{k - 1}$ in $\{1,\ldots,Q\}$ such that if
\[
\yy_i = \xx_{q_i}, \;\; r_i = \|\yy_i\|,
\]
then $(\yy_i)_0^{k - 1}$ is a linearly independent set in $\R^d$ and
\begin{equation}
\label{minkowski}
\prod_{i = 0}^{k - 1} r_i \lesssim_\times 1/Q.
\end{equation}
\end{claim}
\begin{subproof}
We choose the sequence $(q_i)_0^{k - 1}$ recursively. Suppose that $(q_i)_0^{j - 1}$ have been defined for some $j \geq 0$. Let $\Lambda_j = \sum_{i = 0}^{j - 1} \Z \yy_i$, and let $\Delta_j$ denote the Dirichlet fundamental domain of $\Lambda_j$, i.e. the set of points in $\R^d$ which are closer to $\0$ (in the Euclidean metric, which we represent by $\dist$) than to any other point of $\Lambda_j$. Let $V_j = \R\Lambda_j$. Then we choose $q_j\in\{1,\ldots,Q\}$ so as to minimize $\dist(\yy_j,V_j)$, subject to the constraint that $\yy_j\in\Delta_j$. If no value of $q_j$ satisfies this constraint, then we let $k = j$ and stop.

Our first observation is that for $i < j$, since $\yy_i\notin\Delta_j\ni\yy_j$, we have $\yy_i\neq \yy_j \in \Delta_i$, so the definition of $q_i$ implies that $\dist(\yy_i,V_i) \leq \dist(\yy_j,V_i)$. On the other hand, since $\yy_j\in\Delta_j$, we have $\dist(\0,\yy_j) \leq \dist(\pm\yy_i,\yy_j)$. It then follows from a geometric calculation that
\[
\dist(\yy_j,V_{i + 1}) \geq \sqrt{3/4} \dist(\yy_j,V_i),
\]
so $\dist(\yy_j,V_j) \asymp_\times \dist(\yy_j,V_0) = r_j$. This proves that $\yy_j$ is linearly independent of $\yy_0,\ldots,\yy_{j - 1}$. In particular, the recursive construction halts at some stage $k\leq d$.

To prove \eqref{minkowski}, we first compute
\[
\Vol(\Delta_k\cap [-1/2,1/2]^d) \asymp_\times \Vol(V_k\cap \Delta_k) = \prod_{j = 0}^{k - 1} \dist(\yy_j,V_j) \asymp_\times \prod_{j = 0}^{k - 1} r_j.
\]
On the other hand, since the algorithm halted at step $k$, we know that
\[
\xx_q\notin \Delta_k \all q = 1,\ldots,Q.
\]
So the region
\[
(\Delta_k\cap [-1/2,1/2]^d) \times [-Q,Q]\subset \R^{d + 1}
\]
contains no nontrival points of the lattice
\[
\{(\rr + q\xx,q) : \rr\in\Z^d , q\in\Z\} \leq\R^{d + 1}.
\]
Applying Minkowski's theorem completes the proof of \eqref{minkowski}.
\end{subproof}

After reordering, we can without loss of generality assume that the sequence $(q_i)_{i = 0}^{k - 1}$ given by Claim \ref{claimminkowski} satisfies
\begin{equation}
\label{reordering}
q_0/r_0 \leq q_1/r_1 \leq\cdots\leq q_{k - 1}/r_{k - 1}.
\end{equation}
For each $i = 0,\ldots,k - 1$ let $m_i = \lceil 1/(2dr_i) \rceil \asymp_\times 1/r_i$. Fix $j = 0,\ldots,k - 1$, and let
\[
S_j = \{(n_0,\ldots,n_j)\in\Z^{j + 1} : |n_i| < m_i \all i\}.
\]
Then for all $(n_0,\ldots,n_j)\in S_j$, we have
\[
0 < \left\|\sum_{i = 0}^j n_i \yy_i\right\| \leq \sum_{i = 0}^j m_i r_i \leq 1/2
\]
and thus
\[
\sum_{i = 0}^j n_i q_i \xx \equiv \sum_{i = 0}^j n_i \yy_i \not\equiv \0,
\]
so
\[
\sum_{i = 0}^j n_i q_i \neq 0.
\]
It follows that the map $S_j\cap\N^{j + 1} \ni (n_0,\ldots,n_j)\mapsto \sum_0^j n_i q_i \ni \{0,\ldots,\sum_0^j m_i q_i - 1\}$ is injective, so
\[
\#(S_j\cap\N^{j + 1}) = \prod_{i = 0}^j m_i \geq \sum_{i = 0}^j m_i q_i.
\]
Thus by \eqref{reordering},
\[
\prod_{i = 0}^j \frac{1}{r_i} \gtrsim_\times \sum_{i = 0}^j \frac{q_i}{r_i} \asymp_\times \frac{q_j}{r_j}
\]
i.e.
\[
q_j \gtrsim_\times \prod_{i = 0}^{j - 1} \frac{1}{r_i}\cdot
\]
Combining with \eqref{contradictionhypothesis} gives
\[
r_j \geq \kappa q_j^{1 - a} Q^{-A} \gtrsim_\times \kappa Q^{-A} \prod_{i = 0}^{j - 1} r_i^{-(1 - a)}.
\]
Writing $R_j = \log_Q(1/r_j)$, we get
\[
R_j + (1 - a) \sum_{i = 0}^{j - 1} R_i \leq A - \log_Q(\kappa/C_1)
\]
for some constant $C_1 > 0$. Multiplying by $a^{k - 1 - j}$ and summing over $j = 0,\ldots,k - 1$ gives
\[
\sum_{i = 0}^{k - 1} R_i \leq (A - \log_Q(\kappa/C_1))\sum_{i = 0}^{k - 1} a^i;
\]
combining with \eqref{minkowski} gives
\[
1 \leq (A - \log_Q(\kappa/C_2))\sum_{i = 0}^{k - 1} a^i,
\]
where $C_2 > 0$ is a different constant. After choosing $\kappa > C_2$, this is a contradiction to the hypothesis that $A = f_d(a)$.
\end{proof}

\draftnewpage
\section{Proof of Theorem \ref{theorem2}}
\label{sectiongeneral}

The deduction of Theorem \ref{theorem2} from Theorem \ref{theorem1} can be done in a high level of generality, so we recall the following notion:

\begin{definition}[\cite{FishmanSimmons5}]
A \emph{Diophantine space} is a triple $(X,\QQ,H)$, where $X$ is a complete metric space $X$, $\QQ\subset X$ is a dense subset, and $H:\QQ\to(0,\infty)$.
\end{definition}

The prototypical example is the triple $(\R^d,\Q^d,H_\std)$, where $H_\std$ is the standard height function on $\Q^d$, i.e. $H_\std(\pp/q) = q$ whenever $\pp/q\in\Q^d$ is given in reduced form.

\begin{definition}
Let $(X,\QQ,H)$ be a Diophantine space. An \emph{automorphism} of $(X,\QQ,H)$ is a bi-Lipschitz map $\Phi:X\to X$ such that $\Phi(\QQ) = \QQ$ and $H\circ\Phi \asymp_\times H$.

A set $K\subset X$ has the \emph{automorphism property} if for every nonempty open set $B\subset X$, there exists an automorphism $\Phi$ of $X$ such that $\Phi(K)\subset B$.
\end{definition}

\begin{example}
The unit cube $[0,1]^d$ in $(\R^d,\Q^d,H_\std)$ has the automorphism property. The required automorphisms are just affine transformations of $\R^d$ with rational coefficients.
\end{example}

We give another example to illustrate the nontriviality of our definition.

\begin{example}
Let $M$ be a nonsingular rational quadratic hypersurface in projective space $\P_\R^d$, i.e. a set of the form $M = \{[\xx] : \QQ(\xx) = 0\}$ where $\QQ$ is a nondegenerate rational quadratic form on $\R^{d + 1}$ and $[\xx]\in\P_\R^d$ denotes the point corresponding to $\xx\in\R^{d + 1}$. (See \cite{FKMS} for a more detailed exposition.) Then if $\P_\Q^d\cap M\neq\emptyset$, then $(M,\P_\Q^d\cap M,H_\std)$ is a Diophantine space (e.g. \cite[Theorem 8.1(i)]{FKMS}). If $\LL\subset M$ is a rational linear subspace of maximal dimension, then the complement of any neighborhood of $\LL$ has the automorphism property. The required automorphisms are rational projective transformations of $\P_\R^d$ which preserve $\QQ$.
\end{example}

We now state the main result of this section:

\begin{proposition}
\label{propositiongeneral}
Let $(X,\QQ,H)$ be a Diophantine space and let $K\subset X$ have the automorphism property. Suppose that for all $Q > 0$, the set $\{r\in\QQ : H(r)\leq Q\}$ intersects each ball in only finitely many points. Fix $(a,A)\in\R^2$, and suppose that $K$ is not uniformly $\kappa\Psi_{a,A}$-approximable for any $\kappa > 0$. Then the set
\begin{equation}
\label{general}
\{x\in X : \text{$x$ is not $\kappa\Psi_{a,A}$-approximable for any $\kappa > 0$}\}
\end{equation}
is comeager.
\end{proposition}
The terms here should be understood to refer to the obvious generalizations of the corresponding terms in Definition \ref{definitionPsiapproximable} to the setting of Diophantine spaces. We remark that Theorem \ref{theorem2} follows immediately from Proposition \ref{propositiongeneral} and Theorem \ref{theorem1}(ii), since for any function $\Psi$, $[0,1]^d$ is uniformly $\Psi$-approximable if and only if $\R^d$ is.
\begin{proof}
The set \eqref{general} can be written in the form $\bigcap_{\kappa,Q_0\in\N} U_{\kappa,Q_0}$, where $U_{\kappa,Q_0}$ is the set of points which are not $(\kappa\Psi_{a,A},Q_0)$-approximable, i.e.
\[
U_{\kappa,Q_0} = \{x\in X : \exists Q \geq Q_0 \;\; \forall r\in \QQ \text{ if $H(r) \leq Q$ then $\dist(r,x) > \kappa\Psi(H(r),Q)$}\}.
\]
Since for each $Q > 0$, the set $\{r\in\QQ : H(r)\leq Q\}$ has finite intersection with every ball, the intersection over $r\in\QQ$ occurring in this definition can be locally replaced by a finite intersection, so $U_{\kappa,Q_0}$ is open. To complete the proof, it suffices to show that $U_{\kappa,Q_0}$ is dense. Indeed, let $B\subset X$ be a nonempty open set, and let $\Phi:X\to X$ be the automorphism guaranteed by the automorphism property, so that $\Phi(K)\subset B$. Let $C_1 > 0$ be the bi-Lipschitz constant of $\Phi$, and let $C_2 > 0$ be the bound on height distortion (i.e. the implied constant of the asymptotic $H\circ\Phi\asymp_\times H$). Since $K$ is not uniformly $C_1 C_2^{|a| + |A|}\kappa\Psi_{a,A}$-approximable, there exists $x\in K$ such that $x$ is not $(C_1 C_2^{|a| + |A|}\kappa\Psi_{a,A},C_2 Q_0)$-approximable. A calculation shows that $\Phi(x)\in U_{\kappa,Q_0}$. Thus $U_{\kappa,Q_0}\cap B\neq\emptyset$.
\end{proof}

\section{Proof of Theorem \ref{theorem3}}
Our main tool for proving Theorem \ref{theorem3} is the following lemma which is a way of quantifying the fact that the convergents of a real number are the ``best approximations'' to that real number.

\begin{lemma}
\label{lemmaconvergents}
Fix $x\in\R$, and let $(p_n/q_n)_1^\infty$ be the sequence of convergents of $x$. Then for all $p/q\in\Q$, there exists $n\in\N$ such that
\begin{align*}
\left|x - \frac pq\right| &> \frac{1}{2 q_n q_{n + 1}},&
q &> (1/2) q_n.
\end{align*}
\end{lemma}
\comdavid{I think we had a lemma like this in an old version of the Gauss map paper, but I couldn't find it since that was before I moved to OSU.}
Before we begin the proof, we recall (cf. \cite[Theorems 9 and 13]{Khinchin_book}) that for all $n$,
\[
\frac{1}{2 q_n q_{n + 1}} < \left|x - \frac{p_n}{q_n}\right| < \frac{1}{q_n q_{n + 1}}\cdot
\]
\begin{proof}[Proof of Lemma \ref{lemmaconvergents}]
By \cite[Theorem 15]{Khinchin_book}, we may without loss of generality suppose that $p/q$ is an \emph{intermediate fraction}, i.e.
\[
\frac pq = \frac{r p_n + p_{n - 1}}{r q_n + q_{n - 1}}
\]
for some $n\in\N$ and $1\leq r\leq \omega_n$, where $p_{n + 1} = \omega_n p_n + p_{n - 1}$ and $q_{n + 1} = \omega_n q_n + q_{n - 1}$. Since $p/q$ and $x$ lie on opposite sides of $p_{n + 1}/q_{n + 1}$, we get
\[
\left|x - \frac pq\right| > \left|\frac{p_{n + 1}}{q_{n + 1}} - \frac pq\right| = \left|\frac{\omega_n p_n + p_{n - 1}}{\omega_n q_n + q_{n - 1}} - \frac{r p_n + p_{n - 1}}{r q_n + q_{n - 1}}\right|
= \frac{\omega_n - r}{q q_{n + 1}}\cdot
\]
So if $r \leq \omega_n/2$, then
\[
\left|x - \frac pq\right| \geq \frac{r}{q q_{n + 1}} > \frac{1}{2 q_n q_{n + 1}}\cdot
\]
On the other hand, if $r\geq \omega_n/2$, then $q > (1/2) q_{n + 1}$, and
\[
\left|x - \frac pq\right| > \left|x - \frac{p_{n + 1}}{q_{n + 1}}\right| > \frac{1}{2 q_n q_{n + 1}}\cdot
\]
\end{proof}

Now suppose that $x$ is $\alpha\Psi_{a,A}$-approximable for some $a < 1 < \min(A,A + a)$ and $\alpha > 0$. Fix $n\in\N$ large and let $Q = q_{n + 1}/2$. Then there exists $p/q\in\Q$ with $1\leq q\leq Q$ such that
\[
\left|x - \frac pq\right| < \alpha q^{-a} Q^{-A}.
\]
Let $m\leq n$ be chosen so that $(1/2) q_m < q \leq (1/2) q_{m + 1}$. By Lemma \ref{lemmaconvergents},
\[
\left|x - \frac pq\right| > \frac{1}{2 q_m q_{m + 1}}
\]
and thus
\[
\frac{1}{q_m q_{m + 1}} \lesssim_\times \alpha q^{-a} q_{n + 1}^{-A} \lesssim_\times \begin{cases}
\alpha q_m^{-a} q_{n + 1}^{-A} & a\geq 0\\
\alpha q_{m + 1}^{-a} q_{n + 1}^{-A} & a\leq 0
\end{cases}.
\]
Rearranging, we have
\[
\alpha^{-1} q_{n + 1}^A \lesssim_\times
\begin{cases}
q_m^{1 - a} q_{m + 1} & a \geq 0\\
q_m q_{m + 1}^{1 - a} & a \leq 0
\end{cases}
\leq
\begin{cases}
q_n^{1 - a} q_{n + 1} & a \geq 0\\
q_n q_{n + 1}^{1 - a} & a \leq 0
\end{cases},
\]
where the last inequality is due to the assumption $a < 1$. Rearranging again, we get
\[
q_{n + 1} \lesssim_\times \alpha^b q_n^{c - 1}
\]
for all sufficiently large $n$. Since $c\geq 2$, with $\alpha$ small if equality holds, we get $\alpha^b q_n^{c - 1} \geq 2 q_n$ for all sufficiently large $n$. Thus we can apply \cite[Theorem 8.5]{Kleinbock5} to complete the proof.

On the other hand, suppose that $x$ is not $C\alpha^{-b}\psi_c$-approximable. Then by \cite[Theorem 8.5]{Kleinbock5},
\[
q_{n + 1} \lesssim_\times C^{-1} \alpha^b q_n^{c - 1}
\]
for all sufficiently large $n$. Fix $Q\in\N$ large and let $n$ be chosen so that $q_n \leq Q < q_{n + 1}$. First suppose that $a\geq 0$, and let $p/q = p_n/q_n$. Then
\[
q^a Q^A \left|x - \frac pq\right| < q_n^a q_{n + 1}^A \frac{1}{q_n q_{n + 1}} = \frac{q_{n + 1}^b}{q_n^{1 - a}} \lesssim_\times C^{-1/b} \alpha.
\]
On the other hand, if $a\leq 0$, then let $1\leq r\leq \omega_n$ be chosen so that $r q_n + q_{n - 1} \leq Q < (r + 1) q_n + q_{n - 1}$. Then
\[
q^a Q^A \left|x - \frac pq\right| \leq q^a Q^A \left|x - \frac{p_n}{q_n}\right| \asymp_\times Q^{A + a} \frac{1}{q_n q_{n + 1}} \leq \frac{q_{n + 1}^b}{q_n} \lesssim_\times C^{-1/b} \alpha.
\]
Either way, if $C$ is sufficiently large, then $|x - p/q| < \alpha\Psi_{a,A}(q,Q)$. Thus $x$ is $\alpha\Psi_{a,A}$-approximable.

\bibliographystyle{amsplain}

\bibliography{bibliography}

\end{document}